\documentclass{amsart}

\usepackage{helvet, color}

\usepackage{amscd,amsmath,amsxtra,amsthm,amssymb,stmaryrd,xr,mathrsfs,mathtools,enumerate}
\usepackage[all]{xy}

 \DeclareFontFamily{U}{wncy}{}
    \DeclareFontShape{U}{wncy}{m}{n}{<->wncyr10}{}
    \DeclareSymbolFont{mcy}{U}{wncy}{m}{n}
    \DeclareMathSymbol{\Sh}{\mathord}{mcy}{"58}

\newtheorem{theorem}{Theorem}[section]
\newtheorem{lemma}[theorem]{Lemma}

\newtheorem{proposition}[theorem]{Proposition}
\newtheorem{corollary}[theorem]{Corollary}
\newtheorem{definition}[theorem]{Definition}

\newtheorem{remark}[theorem]{Remark}

\numberwithin{equation}{section}

\newcommand{\Gal}{\operatorname{Gal}}

\newcommand{\Zp}{\mathbb{Z}_p}
\newcommand{\ZZ}{\mathbb{Z}}

\newcommand{\ord}{\mathrm{ord}}
\newcommand{\GL}{\mathrm{GL}}
\newcommand{\cH}{\mathcal{H}}

\newcommand{\rhobar}{\bar{\rho}}

\newcommand{\Z}{\mathbb{Z}}
\newcommand{\p}{\mathfrak{p}}
\newcommand{\Q}{\mathbb{Q}}
\newcommand{\F}{\mathbb{F}}

\newcommand{\cL}{\mathcal{L}}

\newcommand{\cO}{\mathcal{O}}
\newcommand{\cS}{\mathcal{S}}

\newcommand{\Sel}{\mathrm{Sel}}
\newcommand{\Char}{\mathrm{char}}

\newcommand{\Gr}{\mathrm{Gr}}

\newcommand{\rank}{\mathrm{rank}}

\newcommand{\corank}{\mathrm{corank}}

\newcommand{\TT}{\mathbf{T}}
\renewcommand{\AA}{\mathbf{A}}
\newcommand{\X}{\mathcal X}
\newcommand{\tor}{\mathrm{tor}}
\newcommand{\cyc}{\mathrm{cyc}}

\begin{document}

\title[Anticyclotomic Selmer groups of positive coranks ]{Comparing anticyclotomic Selmer groups of positive coranks for congruent modular forms}

\author[J.~Hatley]{Jeffrey Hatley}
\address[Hatley]{
Department of Mathematics\\
Union College\\
Bailey Hall 202\\
Schenectady, NY 12308\\
USA}
\email{hatleyj@union.edu}

\author[A.~Lei]{Antonio Lei}
\address[Lei]{D\'epartement de math\'ematiques et de statistique\\
Pavillon Alexandre-Vachon\\
Universit\'e Laval\\
Qu\'ebec, QC, Canada G1V 0A6}
\email{antonio.lei@mat.ulaval.ca}

\begin{abstract}
We study the variation of Iwasawa invariants of the anticyclotomic Selmer groups of congruent modular forms under the Heegner hypothesis. In particular, we show that even if  the Selmer groups we study may have positive coranks, the mu-invariant vanishes for one modular form if and only if it vanishes for the other, and that their lambda-invariants are related by an explicit formula. This generalizes results of Greenberg-Vatsal for the cyclotomic extension, as well as results of Pollack-Weston and Castella-Kim-Longo for the anticyclotomic extension when the Selmer groups in question are cotorsion.
\end{abstract}

\thanks{The second named author's research is supported by the NSERC Discovery Grants Program 05710.}


\maketitle

\section{Introduction}\label{section:intro}

In their seminal paper \cite{greenbergvatsal}, Greenberg and Vatsal showed that if two elliptic curves $E_1$ and $E_2$ over $\Q$ have $p$-torsion subgroups which are isomorphic as $G_\Q = \mathrm{Gal}(\bar{\Q}/\Q)$-modules, then their associated Iwasawa invariants defined over $\Q_\infty$, the cyclotomic $\Zp$-extension of $\Q$, are deeply entangled. In particular, either both of their $\mu$-invariants vanish, or neither does; and in the former case, it is possible to read the $\lambda$-invariant of $E_1$ from the $\lambda$-invariant of the $E_2$. (Actually, \cite{greenbergvatsal} deals only with $p$-ordinary elliptic curves, and this result was generalized to supersingular curves by Kim \cite{kim09}.)

This fact about the variation of Iwasawa invariants was then generalized to Hida families of congruent $p$-ordinary modular forms by Emerton, Pollack, and Weston \cite{epw}, allowing them to prove many cases of the Iwasawa main conjecture. As before, this work was done over $\Q_\infty$, which is the only $\Zp$-extension of $\Q$. In both \cite{greenbergvatsal} and \cite{epw}, the fact that the Selmer groups associated to $p$-ordinary modular forms over $\Q_\infty$ are $\Lambda_\cyc$-{\it cotorsion}, where $\Lambda_\cyc=\Zp[[\Gal(\Q_\infty/\Q)]]$, played a crucial role in their arguments. Indeed, in extending this work to supersingular elliptic curves, Kim needed to use the modified $\pm$-Selmer groups of Kobayashi \cite{kobayashi03}, since the normal Selmer group associated to a $p$-supersingular elliptic curve over $\Q_\infty$ is never cotorsion. Similar results were obtained for $p$-non-ordinary modular forms by the authors \cite{hatleylei}.

In contrast to $\Q$, number fields of degree greater than $1$ can have more than one $\Z_p$-extension. For instance, if $K/\Q$ is an imaginary quadratic field, then the compositum of all $\Z_p$-extensions of $K$ is isomorphic to $\Z_p^2$; among these is the cyclotomic extension, which is the compositum of $K$ and $\Q_\infty$, and the {\it anticyclotomic} extension $K_\infty$, which is pro-dihedral over $\Q$.

Suppose $N>1$ is an integer coprime to the discriminant of $K/\Q$, and factor $N=N^+ N^-$ such that all the primes dividing $N^+$ (resp. $N^-$) are split (resp. inert) in $K$. Then the anticyclotomic Iwasawa theory for modular forms of level $N$ is readily divided into two cases, depending on whether the number of primes dividing $N^{-}$ is even or odd. When this number is odd, the anticyclotomic Iwasawa theory for modular forms closely resembles the cyclotomic Iwasawa theory. In particular, the Greenberg  Selmer group (or Kobayashi-type $\pm$-Selmer groups in the case of $p$-supersingular elliptic curves) of a modular form  is often a $\Lambda$-cotorsion module, where $\Lambda=\Zp[[\Gal(K_\infty/K)]]$, and many of the aforementioned results can be generalized to this setting (see e.g. \cite{pollackweston11,CastKimLon} for this case, or see \cite{kidwell16} for more general number fields).

Throughout this paper, we will instead assume that $N^{-}=1$, and hence it is the product of an {\it even} number of primes. Under this assumption, the Greenberg Selmer group $\Sel(K_\infty,f)$ of a $p$-ordinary modular form $f$ is {\it not} $\Lambda$-cotorsion. In fact, its Pontryagin dual has rank $1$ as a $\Lambda$-module (see Theorem \ref{thm:selmer-double-module}).

The purpose of this paper is to study the variation of Iwasawa invariants in this positive corank setting. Thus, we achieve results in the spirit of Greenberg and Vatsal, but we do so in the absence of a cotorsion hypothesis for the Greenberg Selmer groups associated to our modular forms. Instead, we study various modified Selmer groups, some of which are cotorsion, and some of which have corank $1$. See Section \ref{sec:selmer} for the definitions and basic properties of these modified Selmer groups.

The main technical lemmas which allow us to carry out the classical program in this positive corank setting are contained in Section \ref{sec:comparing}; this section is the heart of the paper, and the results in this section are applicable to $\Lambda$-modules of abitrary rank, so they will surely have applications to Iwasawa theory in other settings.

We study  the vanishing of $\mu$-invariants in Section \ref{sec:mu}, and we study the variation of $\lambda$-invariants in Section \ref{sec:main}. Finally, in Section \ref{sec:kriz-li}, we explain how our results give an algebraic analogue of recent (analytic) results of Kriz and Li \cite{KrizLi}.

\subsection{Notation and Assumptions}\label{sec:notation}

For any field $L$, we let $G_L$ denote the absolute Galois group $\mathrm{Gal}(\bar{L}/L)$. Throughout this article, $p$ denotes a fixed odd prime, and $\rhobar : G_\Q \to \GL_2(k)$ denotes a fixed, absolutely irreducible, $p$-ordinary and $p$-distinguished, modular mod $p$ Galois representation. Let $\bar{N}$ denote the (prime to $p$) Artin conductor of $\rhobar$. We will assume that $\bar{N}$ is square-free. We let $\mathcal{H}$ denote the Hida family of all $p$-ordinary, $p$-stabilized newforms $f$ whose residual mod $p$ Galois representations are isomorphic to $\rhobar$.

Give a modular form $f \in \mathcal H$, let $L_f$ denote the number field generated by its Fourier coefficients, and let $\mathcal O$ denote {the ring of integers of }its completion at a prime $\mathfrak{P}$ above $p$, with uniformizer $\varpi$. When studying two modular forms $f,g \in \mathcal{H}$, we will enlarge $\mathcal O$ as necessary so that it contains the coefficients of both of our modular forms.

Let $K / \Q$ be an imaginary quadratic extension of discriminant $-D_K<0$, coprime to $\bar{N}p$. Write $\bar{N}=\bar{N}^{+}\bar{N}^{-}$, where primes dividing $\bar{N}^{+}$ (resp. $\bar{N}^{-}$) are split (resp. inert) in $K$. We will assume $\bar{N}^- = 1$ and that $p=\mathfrak{p}\bar{\mathfrak{p}}$ splits in $K$. {The assumption that $\bar{N}^-=1$ seems important to our method, as we use it to guarantee the global-to-local maps defining certain Selmer group are surjective; see the discussion following Lemma \ref{lem:Hv-size} as well as the proofs of Lemma \ref{lem:LEO-and-CRK} and Proposition \ref{prop:gr-zero-sel-exact}. It is also present in the assumptions of \cite{LongoVigni}, which we cite in Section \ref{subsec:properties}. It is not immediately clear to the authors whether this assumption can be relaxed to $N^-$ having an even number of prime divisors.}

 Denote by $K_\infty$ the anticyclotomic $\Z_p$-extension of $K$, and for $n\ge0$ we write $K_n$ for the subextension of $K_\infty$ such that $K_n/K$ is of degree $p^n$. Let $\Lambda={\mathcal O}[[\Gamma]]$ be the associated Iwasawa algebra, where $\Gamma :=\mathrm{Gal}(K_\infty / K) \cong \Zp$. We write $\Gamma_n=\Gamma^{p^n}$ for all integers $n\ge0$.

If $\Sigma$ is any finite set of places of $K$ and $F \subset K_\infty$ is an extension of $K$, we denote by $F_\Sigma$ the maximal extension of $F$ unramified outside the places above $\Sigma$. We will write $H^i_\Sigma(F,*)$ for the Galois cohomology $H^i({F_\Sigma/F},*)$.

{Our results will ultimately apply to modular forms $f$ belonging to $\mathcal{H}$ and satisfying a number of additional hypotheses. One suite of hypotheses, labeled \textbf{(admiss.)}, is introduced in Section \ref{subsec:properties} in order to apply results of \cite{LongoVigni}. There is also a mild technical hypothesis \textbf{(H.div)} which is introduced in Section \ref{sec:residual-and-nonprimitive-selmer-groups}; see Remark \ref{remark:hdiv}.}

{We denote by $\mathcal{\hat{H}}$ the subset of $\mathcal H$ of modular forms satisfying all of the above hypotheses.}

\section{{Comparing $\Lambda$-modules of arbitrary rank}}\label{sec:comparing}

Recall that if $M$ is any finitely generated $\Lambda$-module, there is a pseudo-isomorphism
\begin{equation}\label{eq:pseudo}
M\sim \Lambda^{\oplus r}\oplus \bigoplus_{i=1}^s \Lambda/\varpi^{a_i}\oplus \bigoplus_{j=1}^t\Lambda/(F_j)
\end{equation}
for some integers $r,s,t\ge 0$, $a_i\ge1$ and some {distinguished} polynomials $F_j$. {The characteristic ideal of $M$ is defined to be
\[
\Char_\Lambda(M)=\left(\varpi^{\sum_{i=1}^sa_i}\prod_{j=1}^t F_j\right)\Lambda.
\]
}The $\mu$-invariant (respectively $\lambda$-invariant) of $M$ is defined to be $\sum_i a_i$ (respectively $\sum_j \deg(F_j)$) and will be denoted by $\mu(M)$ (respectively $\lambda(M)$). We will also write
\[
\lambda(M)=\lambda(\widehat{M});\quad \mu(M)=\mu(\widehat{M}),
\]
where $\widehat{M}$ is the Pontryagin dual of $M$.

It is well-known that if $A$, $B$ and $C$ are $\Lambda$-torsion modules sitting in a short exact sequence
\[
0\rightarrow A\rightarrow B\rightarrow C\rightarrow 0,
\]
then
\[
\Char_\Lambda(A)\Char_\Lambda(C)=\Char_\Lambda(B).
\]
This implies that
\begin{align*}
\lambda(A) + \lambda(C)& = \lambda(B);\\
\mu(A) + \mu(C)&= \mu(B).
\end{align*}
We will need the following slightly more general result on the additivity of Iwasawa invariants in short exact sequences.

\begin{proposition}\label{prop:additivity}
Given a short exact sequence of finitely-generated $\Lambda$-modules
\[
0 \rightarrow A \rightarrow B \rightarrow C \rightarrow 0,
\]
where $A$ is a torsion module, we have
\[
\lambda(A) + \lambda(C) = \lambda(B)
\]
and
\[
\mu(A) + \mu(C) = \mu(B).
\]
\end{proposition}
\begin{proof}
By the structure theorem of finitely generated $\Lambda$-modules, $B$ is pseudo-isomorphic to
\[
\Lambda^r\oplus B_\tor
\]
for some integer $r\ge 0$.
Since $A$ is torsion, its image in $B$ lies inside $B_\tor$. Therefore, $B/A$ is pseudo-isomorphic to
\[
\Lambda^r\oplus B_\tor/A.
\]
Consequently, our short exact sequence says that
\[
C_\tor\sim B_\tor/A.
\]
Therefore,
\[
\Char_\Lambda( B_\tor)=\Char_\Lambda(A)\Char_\Lambda( C_\tor),
\]
which implies our result.
\end{proof}

\subsection{Results on $\mu$-invariants}

If $G$ is a finite $p$-group, we will write $e(G)$ for the $p$-exponent of $|G|$. Recall that $\Gamma_n$ denotes $\Gamma^{p^n}$ for $n\ge0$, where $\Gamma=\Gal(K_\infty/K)\cong \Zp$.
\begin{lemma}\label{lem:order-M/p}
Let $M$ be a finitely generated $\Lambda$-module and let $r,s,t$ be the integers defined by \eqref{eq:pseudo}. Then,
\[
e\left((M/\varpi)_{\Gamma_n}\right)=(r+s)p^n+O(1).
\]
\end{lemma}
\begin{proof}
We have the pseudo-isomorphism
\[
M/\varpi \sim(\Lambda/\varpi)^{r+s}\oplus \bigoplus_{j=1}^t \Lambda/(F_j,\varpi).
\]
As $\mathcal O$-modules,  $\Lambda/(F_j)\cong \mathcal O^{\deg(F_j)}$. So, we have in fact
\[
\Lambda/(F_j,\varpi)\cong ({\mathcal O}/ \varpi {\mathcal O})^{\deg(F_j)}
\]
as abelian groups. In particular
\[
e\left( \left(\Lambda/(F_j,\varpi)\right)_{\Gamma_n}\right)=O(1)
\]
for all $j$. Therefore, for all $n\ge0$, we have
\[
e\left((M/\varpi)_{\Gamma_n}\right)=(r+s)\times e\left((\Lambda/\varpi)_{\Gamma_n}\right)+O(1),
\]
which finishes the proof.
\end{proof}

This lemma has two immediate and useful corollaries. First, we have a criterion for the vanishing of the $\mu$-invariant of $M$.

\begin{corollary}\label{cor:mu-invariant-zero-condition}
Let $M$ be a finitely generated $\Lambda$-module. Then its $\mu$-invariant vanishes if and only if
\[
e\left((M/\varpi)_{\Gamma_n}\right)=r\times p^n+O(1),
\]
where $r=\rank_{\Lambda}M$.
\end{corollary}

This implies that $\mu$-invariants behave well under congruences of $\Lambda$-modules.

\begin{corollary}\label{cor:mu-vanish-together}
Let $M$ and $M'$ be two finitely generated $\Lambda$-modules such that
\[
\rank_{\Lambda}M=\rank_{\Lambda}M', \quad M/\varpi \cong M'/\varpi
\]
as $\Lambda$-modules. Then the $\mu$-invariant of $M$ vanishes if and only if that of $M'$ does.
\end{corollary}

\subsection{Results on $\lambda$-invariants}

Let us now show that we may also read off the $\lambda$-invariant by considering $\Gamma_n$-coinvariants of certain type of finitely generated $\Lambda$-modules.

\begin{lemma}\label{lem:lambda-inv-from-coinvariants}
Let $F\in \Lambda$ be a distinguished polynomial of degree $d$ and consider the $\Lambda$-module $M=\Lambda/(F)$. Then
\[
e\left((M/\varpi)_{\Gamma_n}\right)=d
\]
when $p^n\ge d$.
\end{lemma}
\begin{proof}
Since $F$ is a distinguished polynomial, we have
\[
M/\varpi=\Lambda/(X^d,\varpi).
\]
Let $\omega_n=(1+X)^{p^n}-1$, where $1+X$ is any topological generator of $\Gamma$. Then,
\[
(M/\varpi)_{\Gamma_n}=\Lambda/(X^d,\omega_n,\varpi).
\]
But the binomial theorem tells us that $(\omega_n,\varpi)=(X^{p^n},\varpi)$. Consequently, when $p^n\ge d$, we have
\[
(M/\varpi)_{\Gamma_n}=\Lambda/(X^d,\omega_n,\varpi)=\Lambda/(X^d,\varpi)=(\cO/\varpi)^d
\]
and hence we are done.
\end{proof}

\begin{corollary}\label{cor:asymptotic-lambda-invariant}
Let $M$ be a $\Lambda$-module isomorphic to $\Lambda^{\oplus r}\oplus \bigoplus_{j=1}^t\Lambda/(F_j)$ for some distinguished polynomials $F_j$ (so in particular, the $\mu$-invariant vanishes and the pseudo-isomorphism \eqref{eq:pseudo} is an isomorphism). If $\lambda(M)$ denotes the $\lambda$-invariant of $M$, then
\[
e((M/\varpi)_{\Gamma_n})=r\times p^n+\lambda(M)
\]
when $n\gg 0$.
\end{corollary}

We will need to consider finitely generated $\Lambda$-modules $M$ with vanishing $\mu$-invariant  for which the pseudo-isomorphism \eqref{eq:pseudo} is injective, but not necessarily surjective. The latter condition is the same as saying that $M$ admits no non-trivial finite submodule. We shall write
\[
N=\Lambda^{\oplus r }\oplus \bigoplus_{j=1}^t\Lambda/(F_j)
\]
for the right-hand side of \eqref{eq:pseudo} and $C$ denotes the (finite) cokernel of $M\rightarrow N$. In particular, we have the short exact sequence
\[
0\rightarrow M\rightarrow N\rightarrow C\rightarrow 0,
\]
from which we deduce the following preliminary lemmas.

\begin{lemma}\label{lem:snake}Let $M,N$ and $C$ be as defined above. Then, there is an exact sequence
$$0\rightarrow C[\varpi]\rightarrow M/\varpi\rightarrow N/\varpi\rightarrow C/\varpi\rightarrow0.$$
\end{lemma}
\begin{proof}
The snake lemma gives the exact sequence
\[
0\rightarrow M[\varpi]\rightarrow N[\varpi]\rightarrow C[\varpi]\rightarrow M/\varpi\rightarrow N/\varpi\rightarrow C/\varpi\rightarrow 0.
\]
But $N[\varpi]=0$ by definition, hence the result.
\end{proof}

In particular, we may consider $C[\varpi]$ as a submodule of $M/\varpi$ and the quotient $(M/\varpi)/C[\varpi]$ makes sense.

\begin{lemma}\label{lem:coinvariants}
Let $M,N$ and $C$ be as in Lemma~\ref{lem:snake}.
Let $n\ge1$ be an integer such that $(C/\varpi)^{\Gamma_n}=\left((M/\varpi)/C[\varpi]\right)^{\Gamma_n}=\left(C/C[\varpi]\right)^{\Gamma_n}=0$, then
\[
e((M/\varpi)_{\Gamma_n})=e((N/\varpi)_{\Gamma_n}).
\]
\end{lemma}
\begin{proof}
On applying the snake lemma to the short exact sequence $$0\rightarrow (M/\varpi)/C[\varpi]\rightarrow N/\varpi\rightarrow C/\varpi\rightarrow0,$$
 our assumption that $(C/\varpi)^{\Gamma_n}=0$ gives the short exact sequence
\[
0\rightarrow \left((M/\varpi)/C[\varpi]\right)_{\Gamma_n}\rightarrow( N/\varpi)_{\Gamma_n}\rightarrow (C/\varpi)_{\Gamma_n}\rightarrow0.
\]
This implies that
\[
e(( N/\varpi)_{\Gamma_n})=e((C/\varpi)_{\Gamma_n})+ e\left(\left((M/\varpi)/C[\varpi]\right)_{\Gamma_n}\right).
\]
Similarly, the assumption that $\left((M/\varpi)/C[\varpi]\right)^{\Gamma_n}=0$ gives the following exact sequence
\[
0\rightarrow (C[\varpi])_{\Gamma_n}\rightarrow (M/\varpi)_{\Gamma_n}\rightarrow \left((M/\varpi)/C[\varpi]\right)_{\Gamma_n}\rightarrow 0
\]
and hence
\[
e((M/\varpi)_{\Gamma_n})=e((C[\varpi])_{\Gamma_n})+ e\left(\left((M/\varpi)/C[\varpi]\right)_{\Gamma_n}\right).
\]
On combining these two equalities, we deduce that
\[
e((M/\varpi)_{\Gamma_n})=e(( N/\varpi)_{\Gamma_n})+e((C[\varpi])_{\Gamma_n})-e((C/\varpi)_{\Gamma_n}).
\]
It remains to show that $e((C[\varpi])_{\Gamma_n})=e((C/\varpi)_{\Gamma_n})$.

Consider the exact sequence
\[
0\rightarrow C[\varpi]\rightarrow C\rightarrow C\rightarrow C/\varpi\rightarrow 0.
\]
The same argument shows us that
\begin{align*}
e(C_{\Gamma_n})&=e\left((C/C[\varpi])_{\Gamma_n}\right)+e((C/\varpi)_{\Gamma_n});\\
e(C_{\Gamma_n})&=e((C[\varpi])_{\Gamma_n})+e\left((C/C[\varpi])_{\Gamma_n}\right).
\end{align*}
Hence, $e((C[\varpi])_{\Gamma_n})=e((C/\varpi)_{\Gamma_n})$ and we are done.
\end{proof}

While the hypothesis on the $\Gamma_n$-invariants in Lemma~\ref{lem:coinvariants} may not hold in general, we show below that {it holds after twisting. Let $\chi$ be a fixed isomorphism of abelian groups $\Gamma\stackrel{\sim}{\longrightarrow}1+p\Zp$}.
Given a $\Lambda$-module $P$ and $i\in\ZZ$, we shall write $P(i)$ for the {the $\Lambda$-module $P\otimes\chi^i$. We may identify $P(i)$ with $P\otimes_{\Zp}\Zp\cdot e_i$, where $e_i$ is some fixed element on which $\Gamma$ acts via $\chi^i$.}

\begin{lemma}\label{lem:twists}
Let $n\ge 0$ be an integer and $P$  a finitely generated $\Lambda$-module. For all but finitely many integers $i$, the {twisted $\Lambda$-module} $P(i)$ satisfies $P(i)^{\Gamma_n}=0$.
\end{lemma}
\begin{proof}
{ Under the identification $P(i)=P\otimes \Zp\cdot e_i$, the $\Lambda$-submodule $P(i)^{\Gamma_n}$ is given by
\[
\left\{x\otimes e_i\in P(i): \gamma_n\cdot x=\chi^{-i}(\gamma_n)x\right\},
\]
where $\gamma_n=\gamma^{p^n}$, which is a topological generator of $\Gamma_n$. It then follows that the direct sum $\bigoplus_{i\in \ZZ} P(i)^{\Gamma_n}\otimes e_{-i}$}
is isomorphic to a $\Lambda$-submodule of $P$. Since $P$ is finitely generated, our result follows.
\end{proof}

The following proposition on $\lambda$-invariants will play an important role in \S\ref{sec:main}, allowing us to compare non-primitive Selmer groups of positive corank.
\begin{proposition}\label{prop:computelambda}
Let $M$ be a finitely generated $\Lambda$-module with vanishing $\mu$-invariant. Furthermore, suppose that $M$ admits no non-trivial finite submodule. Then, for all but finitely many integers $i$ and $n\gg0$,
\[
e((M(i)/\varpi)_{\Gamma_n})=r\times p^n+\lambda(M).
\]
\end{proposition}
\begin{proof}
Let $M,N$ and $C$ be as given in Lemma~\ref{lem:snake} and fix $n$.
Lemma~\ref{lem:twists} tells us that for all but finitely many $i$, we have
$$(C(i)/\varpi)^{\Gamma_n}=\left((M(i)/\varpi)/C(i)[\varpi]\right)^{\Gamma_n}=\left(C(i)/C(i)[\varpi]\right)^{\Gamma_n}=0.$$
Hence, we have
\[
e((M(i)/\varpi)_{\Gamma_n})=e((N(i)/\varpi)_{\Gamma_n})
\]
following Lemma~\ref{lem:coinvariants}.

Note that the $\cO$-linear map induced by $\sigma\otimes e_i\mapsto \chi(\sigma)^{-i}\sigma
$ for $\sigma\in\Gamma$ gives a $\Lambda$-isomorphism from $\Lambda(i)$ to $\Lambda$. Let $F=\sum_{n\ge0}c_n(\gamma-1)^n\in\Lambda$. The image of $F\otimes e_i$ in $\Lambda$ is given by
\[
\sum_{n\ge0}c_n(\chi^{-i}(\gamma)\gamma-1)^n=\sum_{n\ge0}c_n\left(\chi^{-i}(\gamma)(\gamma-1)+\chi^{-i}(\gamma)-1\right)^n.
\]
Since $\chi^{-i}(\gamma)-1\in p\Zp$ and $\chi(\gamma)\in\Zp^\times$, if $F$ is a distinguished polynomial of degree $d$, the image of $F\otimes e_i$ in $\Lambda$ is, up to a unit in $\Zp^\times$, a distinguished polynomial of the same degree.

Recall that $N$ is of the form $\Lambda^{\oplus r }\oplus \bigoplus_{j=1}^t\Lambda/(F_j)$ for some distinguished polynomials $F_j\in \Lambda$. The discussion above tells us that there is a $\Lambda$-isomorphism 
\[
N(i)\cong\Lambda^{\oplus r }\oplus \bigoplus_{j=1}^t\Lambda/(G_j),
\]
where $G_j$ are distinguished polynomials with $\deg(G_j)=\deg(F_j)$.
Therefore,
\[
e((N(i)/\varpi)_{\Gamma_n})=e((N/\varpi)_{\Gamma_n})=r\times p^n+\lambda(M)
\]
for $n\gg0$ thanks to Corollary~\ref{cor:asymptotic-lambda-invariant} and so we are done.
\end{proof}

\section{{Anticyclotomic Selmer Groups}}\label{sec:selmer}

{Let $f \in \cH$ be a modular form of weight $k=2r$ and tame level $\Gamma_0(N)$ as in Section \ref{sec:notation}}.  Following \cite{LongoVigni}, we denote by $T$ the associated $G_\Q$-representation constructed in \cite{nekovar1992}; it is a free rank-2 ${\mathcal O}$-module. Then $V=T \otimes L_{f,\mathfrak{P}}$ is the $r$-twist of the Galois representation constructed by Deligne. Finally, let $A = V/T$. Our choice of normalization makes $V$ self-dual, and after possibly rescaling the isomorphism $V \simeq V^\ast = \mathrm{Hom}_{L_{f,{\mathfrak{P}}}}(V, L_{f,\mathfrak{P}}(1))$, the lattice $T \subset V$ is also self-dual.

In this section, following \cite[Section 2.1]{Castella} closely, we define the various Selmer groups associated to $f$ which we will study.

\subsection{Definitions of modified Selmer groups}\label{subsec:defs} Since $f$ is $p$-ordinary, there exist $G_{\Q_p}$-stable filtrations
\begin{eqnarray}
\label{seq:filtration}&0 \rightarrow V' \rightarrow V \rightarrow V'' \rightarrow 0,\\
\notag &0 \rightarrow T'  \rightarrow T \rightarrow T'' \rightarrow 0,
\end{eqnarray}
where $V'$ and $V''$ are one-dimensional over $L_{f,\mathfrak{P}}$, whereas $T'$ and $T''$ are free of rank one over $\cO$.

Let $F$ be any finite extension of $K$, and for any place $v$ of $F$, define
\[
H^1_f (F_v, V) = \begin{cases} \mathrm{ker}\left( H^1(F_v,V) \rightarrow {H^1(I_v,V)} \right), & v \nmid p, \\
\mathrm{ker}\left( H^1(F_v,V) \rightarrow H^1(F_v,V'') \right), & v \mid p,
\end{cases}
\]
{where $I_v$ denotes the inertia group at $v$. For primes above $p$, $H^1_f(F_v,V)$} is defined by the filtration \eqref{seq:filtration}; we refer to {this subgroup} as the Greenberg local condition. Using the tautological exact sequence
\begin{equation}\label{eq:tautological-exact-sequence}
0 \rightarrow T \rightarrow V \rightarrow A \rightarrow 0,
\end{equation}
we define $H^1_f(F_v,T)$ (resp. $H^1_f(F_v,A)$) as the natural preimage (resp. image) of $H^1_f(F_v,V)$.

Recall that $p=\mathfrak{p} \bar{\mathfrak{p}}$ splits in $K$. Following \cite[Definition 2.2]{Castella}, we define the following Selmer groups.

\begin{definition}\label{def:modified-selmer}
Let $W \in \{V,T,A\}$. For $v \mid p$ and $\mathcal L_v \in \{ \emptyset, \Gr, 0 \}$, set
\[
H^1_{\mathcal L_v}(F_v,W) = \begin{cases}
H^1(F_v,W) & \text{if}\ \mathcal L_v = \emptyset, \\
H^1_f(F_v,W) & \text{if}\ \mathcal L_v = \Gr, \\
\{0\} & \text{if}\ \mathcal L_v = 0.
\end{cases}
\]
For a local condition $\mathcal L_v$ at a place $v$, we write
\[
H^1_{/\mathcal L_v}(F_v,W) = \frac{H^1(F_{v},W)}{H^1_{\mathcal L_v}(F_{v},W)}.
\]
Then for $\mathcal L = \{ \mathcal L_v \}_{v \mid p}$ and $\Sigma$ any finite set of places containing $\{ v \mid Np\infty \}$, we define
\[\label{eq:selmer-def}
\Sel_{\mathcal L}(F,W) = \mathrm{ker} \left( H^1_\Sigma(F,W) \rightarrow \prod_{v \nmid p} H^1_{/f}(F_v,W) \times \prod_{v \mid p} H^1_{/ \mathcal L_v}(F_v,W) \right)
\]
\end{definition}

\begin{remark}\label{rmk:local-subgroups}
If $\mathcal L = \{\Gr, \Gr\}$, we will omit the subscripts and just write (for example) $\Sel(F,W)$.
\end{remark}

Let $\Psi : \Gamma \hookrightarrow \Lambda^\times$ denote the map sending $\gamma \in \Gamma$ to the corresponding group-like element in $\Lambda^\times$, and let $\Lambda(\Psi^{-1})$ denote the free rank-1 $\Lambda$-module with $\Gamma$ action by $\Psi^{-1}$. Then we define the $\Lambda$-modules
\[
\TT = T \otimes_{\Z_p} \Lambda(\Psi^{-1}) \quad \text{and} \quad \AA = \TT \otimes_{\Lambda} \mathrm{Hom}_{\Z_p}(\Lambda, \Q_p / \Z_p),
\]
where $\Gamma$ acts on $\TT$ diagonally. We may now define $\Sel_{\mathcal L}(K,\TT)$ and $\Sel_{\mathcal L}(K,\AA)$ in a manner similar to Definition \ref{def:modified-selmer}, and as noted on \cite[p. 9]{Castella} we have
\[
\Sel_{\mathcal L}(K,\TT) \simeq \varprojlim \Sel_{\mathcal L}(F,T) \quad \text{and} \quad \Sel_{\mathcal L}(K,\AA) \simeq \varinjlim \Sel_{\mathcal L}(F,A),
\]
where the limits are taken with respect to the corestriction and restriction maps, respectively, over all finite extensions $K \subset F \subset K_\infty$.

{Thus, for a prime $w$ of $K_\infty$, we may explicitly describe the nontrivial local conditions defining the Selmer group $\Sel_{\mathcal L}(K,\AA)$. For $w \nmid p$ we have the ``unramified'' local condition
\[
H^1_f (K_{\infty,w}, A) = \mathrm{ker}\left( H^1(K_{\infty,w},A) \rightarrow H^1(I_w,A) \right).
\]
To describe the Greenberg local condition for $w \mid p$, note that the tautological exact sequence \eqref{eq:tautological-exact-sequence} induces a $G_{\Q_p}$-stable filtration
\begin{equation}\label{eq:filtration-for-A}
0 \rightarrow A' \rightarrow A \rightarrow A'' \rightarrow 0,
 \end{equation}
 and then the Greenberg condition is
 \[
H^1_f (K_{\infty,w}, A) =  H^1(K_{\infty,w},A').
 \]
We will use this explicit description in the proof of Proposition \ref{prop:non-primitive-torsion-compatibility}.
 }

\subsection{Local Galois cohomology subgroups}\label{sec:local-galois-cohomology}

Following Kidwell \cite{kidwell16}, for any place $v$ of $K$, we define
\[
 H_v = \varinjlim \displaystyle\prod_{w \mid v} H^1_{/f}(F_w,A),
\]
where $F$ {runs} over the finite extensions of $K$. A careful study of these groups is made in \cite[Propositions 4.2 and 4.3]{kidwell16}, which we summarize below. For the definitions of $\mu$- and $\lambda$-invariants, see the beginning of Section \ref{sec:comparing}.

\begin{lemma}
Let $v \in \Sigma$ be a place of $K$.
\label{lem:Hv-size}
\begin{enumerate}[(i)]
\item If $v \nmid p$ is finitely decomposed in $K_\infty$, then
\[
 H_v \simeq \prod_{w \mid v} H^1(K_{\infty,w},A)
\]
as $\Lambda$-modules, and $ H_v$ is a cofinitely generated, cotorsion $\Lambda$-module with $\mu$-invariant zero and $\lambda$-invariant
\[
\sum_{w \mid v} \mathrm{corank}_{\mathcal O}(H^1({K}_{\infty,w},A)).
\]
\item If $v \mid p$, then $H_v$ is a cofinitely generated $\Lambda$-module with $\Lambda$-corank $1$ and $\mu$-invariant zero.
\item If $v \nmid p$ splits completely in $K_\infty$, then $H^1(K_{\infty,{w}},A) \simeq H^1(K_v,A)$ and $H^1_{/f}(K_v,A)$ are cofinitely generated $\mathcal O$-modules, with
\[
\mathrm{corank}_{\mathcal O}(H^1_{/f}(K_v,A)) = \mathrm{rank}_{\mathcal O}(H^0(K_v,A^\ast)).
\]
If we have an isomorphism of $\mathcal O$-modules
\[
H^1_{/f}(K_v,A) \simeq (L_{f,{\mathfrak{P}}} / \mathcal O)^r \oplus \sum_{i=1}^t {\mathcal O}/ \varpi^{m_i}
\]
for some $r \geq 0$ and $m_i \geq 0$, then
\[
H_v \simeq \widehat{\Lambda}^r \oplus \sum_{i=1}^t \widehat{\Lambda / \varpi^{m_i}\Lambda},
\]
where \ $\widehat{}$ \ denotes Pontryagin duality. Thus $ H_v$ is a cofinitely-generated $\Lambda$-module with $\Lambda$-corank given by $\mathrm{rank}_{\mathcal O}(H^0(K_v,A^\ast))$, $\mu$-invariant $\sum_{i=1}^t m_i$, and $\lambda$-invariant zero.
\end{enumerate}
\end{lemma}
\begin{proof}
See \cite[Propositions 4.2 and 4.3]{kidwell16}.
\end{proof}

For many applications in Iwasawa theory, it is important to show that the global-to-local maps defining the various Selmer groups are surjective. One consequence of Lemma \ref{lem:Hv-size} is that primes $v$ which split completely in $K_\infty$ pose an obstruction to such surjectivity.

However, by \cite[Corollary 1]{brink}, our hypothesis that $N^- = 1$ implies that, if we take $\Sigma$ as small as possible, then case (iii) never occurs for any $v \in \Sigma$, and so every $ H_v$ which occurs in the definitions of our Selmer groups is cofinitely-generated with $\mu$-invariant zero, and that $H_v$ is cotorsion except when $v \mid p$.

\subsection{Some structure theorems for modified Selmer groups}\label{subsec:properties}
We now collect some properties of the modified Selmer groups which we will need in the rest of our paper. Recall that our modular form $f$ has a residual Galois representation  $\rhobar_f$ isomorphic to some fixed representation $\rhobar$. Write $\X_{\mathcal L}(K,{\AA})$ for the Pontryagin dual of $\Sel_{\mathcal L}(K,{\AA})$, and for any torsion $\Lambda$-module $M$, write $\Char_\Lambda(M)$ for the characteristic ideal of $M$. Finally, let $\iota : \Lambda \to \Lambda$ denote the involution induced by inversion in $\Gamma_\infty$.

We first recall a structure theorem for (the dual of) the full Selmer group $\Sel(K,{\AA})$. (Recall that when $\mathcal L = \{ \Gr,\Gr \}$, we omit the subscript in $\Sel_{\mathcal{L}}$.) {In order to do so, we will impose some an additional hypotheses for the rest of the paper. Let $h_K$ denote the class number of $K$, and let $\phi$ denote the Euler totient function. Following \cite{LongoVigni}, we assume}

{
\[
\textbf{(admiss.)}
\begin{cases}
\bullet \ \text{The modular form $f$ is of \underline{even} weight $k=2r$ and \underline{squarefree} level $N$.} \\
\bullet \ \text{$p$ does not ramify in $L_f$.} \\
\bullet \ \text{$p \nmid 6N(k-1)!\phi(N) h_K$} \\
\end{cases}
\]
}

{\begin{remark} We make this assumption in order to apply the results of \cite{LongoVigni}, where it is built into \cite[Definition 2.1]{LongoVigni} as part of the condition for a modular form $f$ (and corresponding imaginary quadratic field $K$ and prime $p$) to be {\it admissible}. In particular, the hypothesis that $p$ does not ramify in $L_f$ is mostly to simplify some of the calculations carried out and does not seem to be a crucial assumption.
\end{remark}
}

\begin{theorem}\label{thm:selmer-double-module}
There exists a finitely-generated torsion $\Lambda$-module $M$ such that $\Char_\Lambda(M) = \Char_\Lambda(M)^\iota$ and
\[
\X(K,{\AA}) \sim \Lambda \oplus M \oplus M.
\]
\end{theorem}
\begin{proof}
Recall that $f$ has even weight $k=2r$. When $r=1$, this is \cite[Theorem B(b,c)]{howard2}, and when $r > 1$, this is \cite[Theorem 1.1]{LongoVigni}.
\end{proof}

We now collect properties of the modified Selmer groups.

\begin{lemma}\label{lem:selmer-ranks}
If $\rhobar_f \vert_{\mathrm{Gal}(\bar{\Q}/K)}$ is absolutely irreducible, then
\begin{enumerate}[(i)]
\item $\mathrm{rank}_\Lambda \Sel(K, \TT) = \rank_\Lambda \X(K,\AA)=1$.
\item $\mathrm{rank}_\Lambda \X_{\Gr,\emptyset}(K,\AA) = 1 + \mathrm{rank}_\Lambda \X_{\Gr,0}(K,\AA)$.
\item $\mathrm{char}(\X_{\Gr,\emptyset}(K,\AA)_{\mathrm{tors}}) = \mathrm{char}(\X_{\Gr,0}(K,\AA)_{\mathrm{tors}})$ up to powers of $p\Lambda$.
\end{enumerate}
\end{lemma}
\begin{proof}
\begin{enumerate}[(i)]
\item When $r=1$, this is \cite[Corollary 3.3.4]{howard2}, and when $r > 1$ this is \cite[Theorem 3.5]{LongoVigni}. For the latter, a similar result is proved in \cite[Theorem~1.1]{BL}.
\item and (iii) These are proven in \cite[Lemma 2.3]{Castella} in the case where $f$ corresponds to an elliptic curve, but the argument relies purely on formal Galois cohomological results from \cite{MazurRubin} and extends to higher weights without change.
\end{enumerate}
\end{proof}

Although $\Sel(K,\AA)$ is not $\Lambda$-cotorsion, the next lemma shows that some of the modified Selmer groups are more amenable to classical arguments.

\begin{lemma}\label{lem:regular-and-modified-selmer-equality}
We have an equality of Selmer groups
\[
\Sel(K,\TT) = \Sel_{\Gr,\emptyset}(K,\TT).
\]
In addition, the modules $\X_{\Gr,0}(K,\AA)$ and $\X_{\emptyset,0}(K,\AA)$ are $\Lambda$-torsion.
\end{lemma}
\begin{proof}
First note that by Theorem \ref{thm:selmer-double-module} and Lemma \ref{lem:selmer-ranks} we have $\mathrm{rank}_\Lambda \Sel(K,\TT)=1$. The result now follows from the proofs of Lemmas A.2 and A.4 of \cite{Castella}.
\end{proof}
Let us summarize the statements on the  coranks of the various Selmer groups we study.
\begin{corollary}\label{cor:coranks}
We have the following formula:
\[
\rank_\Lambda\X_\cL(K,\AA)=
\begin{cases}
0&\text{if }\cL=\{\emptyset,0\},\{\Gr,0\},\\
1&\text{if }\cL=\{\Gr,\Gr\},\{\Gr,\emptyset\}.
\end{cases}
\]
\end{corollary}

\subsection{{Finite index submodules of modified Selmer groups}}

When comparing Iwasawa invariants of isomorphic $\Lambda$-modules (whose structures are only known, {\it a priori}, up to pseudo-isomorphism), it is important to know whether these modules have any non-trivial finite $\Lambda$-submodules. For example, the reader is suggested to see \cite[Theorem 4.1.1]{epw} in the cyclotomic case and the proof of \cite[Proposition 3.6]{pollackweston11} in the anticyclotomic case where $N^-$ is the product of an even number of primes.

{To obtain similar results in our present setting, we will need to apply a result of Greenberg. Before stating the relevant result, let us recall some of the notation used in \cite{Greenberg16}.
First, to bring our notation in line with that of \textit{loc cit}, we have $\mathcal D = \AA$ and $\mathcal T = \TT$, and write $\mathcal T^\ast = \mathrm{Hom}(\mathcal D, \mu_{p^\infty})$. Let $\mathfrak{m}$ denote the maximal ideal of $\Lambda$. Greenberg also writes $Q_\cL(K,\mathcal D)$ for the target of the map  defining $\Sel_\cL(K,\AA)$. That is, we define
\[
Q_\cL(K,\mathcal D) := \prod_{w } H_w \times \prod_{w \mid p} H^1_{/\cL_w}(K_{\infty, w}, A)
\]
so that
\[
\Sel_\cL(K,\AA) = \mathrm{ker} \left( H^1_\Sigma (K_\infty, A) \rightarrow Q_\cL(K,\mathcal D) \right).
\]
Greenberg introduces the following hypotheses:
\begin{itemize}
\item $\mathrm{RFX}(\mathcal D)$: The module $\mathcal T$ is reflexive as a $\Lambda$-module.
\item $\mathrm{LOC}_v^{(1)}(\mathcal D)$: $(\mathcal T^\ast)^{G_{K_v}}=0$ for $v \in \Sigma$.
\item $\mathrm{LOC}_v^{(2)}(\mathcal D)$: The $\Lambda$-module $\mathcal T^\ast /(\mathcal T^\ast)^{G_{K_v}}$ is reflexive for $v \in \Sigma$.
\item $\mathrm{LEO}(\mathcal D)$: The discrete, co-finitely generated $\Lambda$-module
\[
\Sh(K,\Sigma, \mathcal D) = \mathrm{ker}\left( H^2(K_\Sigma / K, \mathcal D) \rightarrow \prod_{v \in \Sigma} H^2(K_v, \mathcal D)  \right)
\]
is cotorsion.
\item $\mathrm{CRK}(\mathcal D, \mathcal L)$: We have an equality of coranks
\[
\mathrm{corank}_\Lambda \left(H^1(K_\Sigma / K, \mathcal D)\right) = \mathrm{corank}_\Lambda \left(\Sel_\cL(K,\mathcal D)\right) + \mathrm{corank}_\Lambda \left( Q_\cL(K,\mathcal D) \right).
\]
\end{itemize}
}

{Greenberg calls a $\Lambda$-module $M$ \textit{almost divisible} if $PM=M$ for almost all height one prime ideals $P$ in $\mathrm{Spec}(\Lambda)$. In particular, an almost divisible $\Lambda$-module has no proper finite-index $\Lambda$-submodules. Greenberg also introduces the notion of a set of local conditions $\cL=\{\cL_v\}$ being almost-divisible; this just means that the corresponding local cohomology quotient group $H^1_{/\cL_v}(K_v,\mathcal D)$ is an almost divisible $\Lambda$-module for each place $v$.}

{Since they will be useful again in Section \ref{sec:main}, let us establish the following two lemmas. Recall that $\mathcal D = \AA$ and $\mathcal T = \TT$.}

\begin{lemma}\label{lem:no-subquotient-mup}
{
The module $\mathcal D[\mathfrak{m}]$ has no subquotient isomorphic to $\mu_p$ for the action of $G_K$.
}
\end{lemma}
\begin{proof}
{
 By the discussion in \cite[Section 4.3.3.]{Greenberg16}, this follows from our assumption that the residual representation is irreducible.
}
\end{proof}

\begin{lemma}\label{lem:LEO-and-CRK}
{
The modified Selmer groups $\Sel_{\Gr,\emptyset}(K,\AA)$ and $\Sel_{{\emptyset,0}}(K,\AA)$ satisfy $\mathrm{LEO}(\mathcal D)$ and $\mathrm{CRK}(\mathcal D, \mathcal L)$.
}
\end{lemma}
\begin{proof}
{
In light of Corollary \ref{cor:coranks} and Lemma \ref{lem:Hv-size}, the fact that $\mathrm{LEO}(\mathcal D)$ holds follows from the discussion in \cite[Section 4.3.2]{Greenberg16}. Finally, to verify $\mathrm{CRK}(\mathcal D, \mathcal L)$ we must check that
\[
2=\corank_\Lambda H^1(K_\Sigma/K,\AA)=\corank_\Lambda \Sel_\cL(K,\AA)+\corank_\Lambda Q_\cL(K,\AA).
\]
We use Lemma \ref{lem:Hv-size}. In the case of $\mathcal L=\{{\emptyset,0}\}$ we have
\[
\corank_\Lambda \Sel_\cL(K,\AA)=0,\ \corank_\Lambda Q_\cL(K,\AA)=2+0.
\]
Whereas for $\mathcal L=\{\Gr,\emptyset\}$ we have
\[
\corank_\Lambda \Sel_\cL(K,\AA)=1,\ \corank_\Lambda Q_\cL(K,\AA)=1+0.
\]
This completes the proof.
}
\end{proof}

{The main proposition that we need is the following specialization of \cite[Proposition 4.1.1]{Greenberg16}:}
\begin{proposition}\label{prop:greenberg-finite-index}
{Suppose that $\mathrm{RFX}(\mathcal D)$ and $\mathrm{LEO}(\mathcal D)$ are both satisfied, that $\mathrm{LOC}_v^{(2)}(\mathcal D)$ is satisfied for all $v \in \Sigma$, and that there exists a non-archimedean prime $v \in \Sigma$ such that $\mathrm{LOC}_v^{(1)}(\mathcal D)$ is satisfied. Suppose also that the set of local conditions $\cL$ is almost divisible, that $\mathrm{CRK}(\mathcal D, \cL)$ is satisfied, and that $\mathcal{D}[\mathfrak{m}]$ has no subquotients isomorphic to $\mu_p$ for the action of $G_K$.}

{Then $\Sel_\cL(K,\mathcal D)$ is an almost divisible $\Lambda$-module.}
\end{proposition}

We may now apply Greenberg's proposition to obtain the following result.

\begin{proposition}\label{prop:modified-selmer-no-finite-index}
The modified Selmer groups $\Sel_{\Gr,\emptyset}(K,\AA)$ and $\Sel_{{\emptyset,0}}(K,\AA)$ have no proper finite-index $\Lambda$-submodules.
\end{proposition}
\begin{proof}
{This follows upon checking that these Selmer groups satisfy the hypotheses of Proposition \ref{prop:greenberg-finite-index}. Recall that $\mathcal D = \AA$ and $\mathcal T = \TT$.  The hypothesis $\mathrm{RFX}(\mathcal D)$ is satisfied since $\TT$ is free of rank 2 over $\Lambda$.  That the local conditions under consideration are almost divisible is clear from Lemma \ref{lem:Hv-size}. We verified $\mathrm{LEO}(\mathcal D)$ and $\mathrm{CRK}(\mathcal D, \mathcal L)$ for these choices of local conditions in Lemma \ref{lem:LEO-and-CRK}. The hypothesis on $\mathcal{D}[\mathfrak{m}]$ was verified in Lemma \ref{lem:no-subquotient-mup}.}

{The hypothesis $\mathrm{LOC}_v^{(1)}(\mathcal D)$ is satisfied by at least one $v\in \{\mathfrak{p}, \bar{\mathfrak{p}} \} \subset  \Sigma$ by \cite[Lemma 5.2.2]{Greenberg11}. The same lemma shows that $\mathrm{LOC}_v^{(1)}$ holds for any prime $v \in \Sigma$ which does not split completely in $K_\infty$, and $\mathrm{LOC}_v^{(1)}$ implies $\mathrm{LOC}_v^{(2)}$. Thus, if $\Sigma$ is as small as possible, then the discussion in Section \ref{sec:local-galois-cohomology} shows that this handles all $v \in \Sigma$, and this completes the proof.}
\end{proof}

\subsection{Residual and Non-primitive Selmer Groups}\label{sec:residual-and-nonprimitive-selmer-groups}

Our eventual goal is to compare Selmer groups for modular forms whose residual Galois representations $A[\varpi]$ are isomorphic. We begin by introducing auxiliary Selmer groups which were first defined by Greenberg and Vatsal \cite{greenbergvatsal}; they are formed by relaxing the local conditions at some finite primes away from $p$.

Let $\bar{T}=T/\varpi T$ denote the two-dimensional $k$-vector space on which $G_\Q$ acts via the absolutely irreducible Galois representation $\bar{\rho}$ underlying our fixed Hida family $\mathcal H$; thus, in the notation of the previous section, $A[\varpi]\simeq \bar{T}$.

\begin{definition}
Let $\Sigma_0 \subset \Sigma$ be a subset of $\Sigma$ not containing the archimedean primes or the primes above $p$, and let $\mathcal L$ be a set of local conditions. For any finite extension $F / K$ and for a representation $W \in \{T, A\}$, we define the {\it non-primitive Selmer groups}
\[\label{eq:nonprim-selmer-def}
\Sel_{{\mathcal L}}^{\Sigma_0}(F,W) = \mathrm{ker} \left( H^1_\Sigma(F,W) \rightarrow \prod_{v  \in \Sigma \setminus \Sigma_0} H^1_{/{\mathcal L_v}}(F_v,W) \right),
\]
and as before we may define
\[
\Sel_{{\mathcal L}}^{\Sigma_0}(K, \AA) \simeq \varinjlim \Sel^{\Sigma_0}(F,A).
\]
In the obvious way, we also define the associated {\textit residual} non-primitive Selmer groups
\[\label{eq:nonprim-residual-selmer-def}
\Sel_{{\mathcal L}}^{\Sigma_0}(F,\bar{T}) = \mathrm{ker} \left( H^1_\Sigma(F,\bar{T}) \rightarrow \prod_{v  \in \Sigma \setminus \Sigma_0} H^1_{/{\mathcal L_v}}(F_v,\bar{T})\right)
\]
and
\[
\Sel_{{\mathcal L}}^{\Sigma_0}(K, \AA[\varpi]) \simeq \varinjlim \Sel^{\Sigma_0}(F,\bar{T}).
\]
\end{definition}

The non-primitive Selmer groups are compatible with taking $\varpi$-torsion under some divisibility hypothesis in the following way. {Recall the filtration \eqref{eq:filtration-for-A}
\[
0 \rightarrow A' \rightarrow A \rightarrow A'' \rightarrow 0.
\]
We make the following hypothesis.}

\noindent
\textbf{(H.div)} For all primes $w$ of $K_\infty$ above $p$, both $A^{G_{K_{\infty,w}}}$ and $(A')^{G_{K_{\infty,w}}}$ are divisible.

\bigskip

\begin{proposition}\label{prop:non-primitive-torsion-compatibility}
If $\Sigma_0$ contains all the primes at which $A$ is ramified, then we have an injection  of $\Lambda$-modules
\[
\Sel_{{\mathcal L}}^{\Sigma_0}(K,\AA[\varpi])\hookrightarrow\Sel_{{\mathcal L}}^{\Sigma_0}(K,\AA)[\varpi]
\]
where the image is of finite index.
Suppose that \textbf{(H.div)} holds, then we have an isomorphism.
\end{proposition}
\begin{proof}
 By the absolute irreducibility of $\bar{T}$, the short exact sequence
\begin{equation}\label{eq:short-exact-torsion}
0 \rightarrow \bar{T} \rightarrow A \xrightarrow{\varpi} A \rightarrow 0
\end{equation}
induces an isomorphism
\[
H^1_\Sigma(K_\infty,\bar{T}) \simeq H^1_\Sigma(K_\infty,A)[\varpi].
\]
So, we must check compatibility of the local conditions defining the relevant Selmer groups.

For a prime $w$ of $K_\infty$ dividing $\ell \in \Sigma \setminus \Sigma_0$ with $\ell \neq p$, since $A$ is not ramified at $w$, the corresponding inertia group $I_w$ acts trivially on $A$. The long exact sequence coming from \eqref{eq:short-exact-torsion} yields the exact sequence
\[
0\rightarrow A/\varpi A\rightarrow H^1(I_w,\bar{T}) \rightarrow H^1(I_w,A)[\varpi] \rightarrow 0.
\]
Since $A$ is divisible, the first term in this sequence is zero, so we have an isomorphism
\[
H^1(I_w,\bar{T}) \simeq H^1(I_w,A)[\varpi],
\]
as desired.

 For the Greenberg condition, we have the exact sequence
\[
0\rightarrow (A')^{G_{K_{\infty,w}}}/\varpi (A')^{G_{K_{\infty,w}}}\rightarrow H^1_f(K_{\infty,w},\bar{T}) \rightarrow H^1_f(K_{\infty,w},A)[\varpi] \rightarrow 0.
\]
Since $(A')^{G_{K_{\infty,w}}}$ is cofinitely generated over $\cO$, the quotient $(A')^{G_{K_{\infty,w}}}/\varpi (A')^{G_{K_{\infty,w}}}$ is finite.
Similarly, we have
\[
0\rightarrow A^{G_{K_{\infty,w}}}/\varpi A^{G_{K_{\infty,w}}}\rightarrow H^1(K_{\infty,w},\bar{T}) \rightarrow H^1(K_{\infty,w},A)[\varpi] \rightarrow 0.
\]
This gives an exact sequence
\[
0\rightarrow \Sel_{{\mathcal L}}^{\Sigma_0}(K,\AA[\varpi])\rightarrow\Sel_{{\mathcal L}}^{\Sigma_0}(K,\AA)[\varpi]
\rightarrow \prod_{w|p} Q_w,
\]
where $Q_w$ is either $(A')^{G_{K_{\infty,w}}}/\varpi (A')^{G_{K_{\infty,w}}}$, $ A^{G_{K_{\infty,w}}}/\varpi A^{G_{K_{\infty,w}}}$ or $0$ depending on the choice of $\cL$.
See \cite[proof of Lemma~3.5]{pollackweston11} for a similar calculation. The last term in the exact sequence is finite and is zero when \textbf{(H.div)} holds, hence the result.
\end{proof}

\begin{remark}\label{remark:hdiv}
We remark that if ${\bar T}^{G_{K_{v}}}=0$, then ${\bar T}^{G_{K_{\infty,w}}}=0$ for $w|v$ since $K_{\infty,w}/K_v$ is a pro-$p$ extension.  This in turn implies that $A^{G_{K_{\infty,w}}}=0$ and in particular \textbf{(H.div)} holds. In the case where $f$ corresponds to an elliptic curve, it is the same as saying that it has no $p$-torsion defined over $K_v$.

If we are only interested in a particular choice of $\cL$, we may weaken \textbf{(H.div)} and only impose the relevant divisibility condition for the local conditions we consider. For example, if $\cL=\{\Gr,\Gr\}$, we only need to assume that $(A')^{G_{K_{\infty,w}}}$ is divisible for all $w|p$. If $\cL=\{\emptyset,0\}$, then we only need to assume that $A^{G_{K_{\infty,w}}}$ is divisible for all $w$ dividing the prime above $p$ where we impose the $\emptyset$ condition.
\end{remark}

From now on, we assume that all the modular forms we consider satisfy \textbf{(H.div)}.
The main utility of the residual non-primitive Selmer groups is given by the following proposition.

\begin{proposition}\label{prop:residual-selmer-isomorphic}
Let $f$ and $g$ be modular forms whose associated residual Galois representations $\bar{T}_f$ and $\bar{T}_g$ are isomorphic. If $\Sigma_0$ contains {all primes for which $T_f$ or $T_g$ is ramified}, then there is an isomorphism of $\Lambda$-modules
\[
\Sel_{{\mathcal L}}^{\Sigma_0}(K,\AA_f[\varpi]) \simeq \Sel_{{\mathcal L}}^{\Sigma_0}(K,\AA_g[\varpi]).
\]
\end{proposition}

\begin{proof}
It suffices to show that the local conditions defining $\Sel_{{\mathcal L}}^{\Sigma_0}(K,\AA_f[\varpi])$ depend only on the isomorphism class of $\bar{T}_f$. This is clear from our proof of Proposition~\ref{prop:non-primitive-torsion-compatibility}.
\end{proof}

\section{The vanishing of $\mu$-invariants}\label{sec:mu}

In this section, we study the vanishing of $\mu$-invariants of the various Selmer groups we defined in \S\ref{sec:selmer}.  Recall the modified and non-primitive Selmer groups which we defined in Section~\ref{sec:selmer}{, which depend on a choice $\cL$ among the four local conditions $\{\emptyset,0\}, \{\Gr,0\}, \{\Gr,\Gr\}, \{\Gr,\emptyset\}$}.

\begin{lemma}\label{lem:compareprimitive}
Let $\cL$ be any choice of local conditions and $\Sigma_0 \subset \Sigma$ be a subset of $\Sigma$ not containing the archimedean primes or the primes above $p$, then
\[
\mu( \Sel_{\mathcal L}(K,\AA))=\mu(\Sel^{\Sigma_0}_{\mathcal L}(K,\AA));\quad \corank_\Lambda \Sel_{\mathcal L}(K,\AA)=\corank_\Lambda\Sel^{\Sigma_0}_{\mathcal L}(K,\AA).
\]
\end{lemma}
\begin{proof}
We have the exact sequence
\begin{equation}\label{eq:nonprimitive-selmer-left-exact}
0 \rightarrow \Sel_{\mathcal L}(K,\AA) \rightarrow \Sel^{\Sigma_0}_{\mathcal L}(K,\AA) \rightarrow \prod_{v \in \Sigma_0} H_v
\end{equation}
by definitions. Lemma~\ref{lem:Hv-size}(i) tells us that {for all $v\in\Sigma_0$, the $\Lambda$-module $H_v$ is cotorsion with $\mu$-invariant equal to zero}. Hence the result follows on taking duals in \eqref{eq:nonprimitive-selmer-left-exact} and applying Proposition~\ref{prop:additivity}.
\end{proof}

{Let us introduce some additional notation to be used throughout the remainder of the paper. Given a modular form $f$ and a choice $\cL$ of local conditions, write $\mu_{\mathcal L}(f)$ and $\lambda_{\mathcal L}(f)$ for the $\mu$- and $\lambda$-invariants of $\Sel_{\mathcal L}(K,\AA_f)$, respectively. Given two modular forms $f$ and $g$, let $\cS_{\mathcal L}(f)$ and $\cS_{\mathcal L}(g)$ denote the non-primitive Selmer groups $\Sel^{\Sigma_0}_{\mathcal L}(K,\AA_f)$ and $\Sel^{\Sigma_0}_{\mathcal L}(K,\AA_g)$, respectively, where $\Sigma_0 \subset \Sigma$ is a finite set of primes containing all primes for which $T_f$ or $T_g$ is ramified.}

\begin{proposition}\label{prop:mucong}
Let $f$ and $g$ be modular forms satisfying the hypotheses of Section \ref{sec:notation} whose associated residual Galois representations $\bar{T}_f$ and $\bar{T}_g$ are isomorphic.  Then the $\mu$-invariant of $\cS_\cL(f)$ vanishes if and only if that of $\cS_\cL(g)$ does.
\end{proposition}
\begin{proof}
Since $\bar{T}_f \simeq \bar{T}_g$, Propositions \ref{prop:non-primitive-torsion-compatibility} and \ref{prop:residual-selmer-isomorphic} give the following $\Lambda$-module isomorphisms
\[
\cS_{\mathcal L}(f)[\varpi] \simeq \Sel^{\Sigma_0}_{\mathcal L}(K,\AA_f[\varpi])
\simeq \Sel^{\Sigma_0}_{\mathcal L}(K,\AA_g[\varpi]) \simeq \cS_{\mathcal L}(g)[\varpi].
\]
Furthermore,  $\cS_\cL(f)$ and $\cS_\cL(g)$ have the same $\Lambda$-corank on combining Corollary~\ref{cor:coranks} and the second half of Lemma~\ref{lem:compareprimitive}. Therefore, our result follows from Corollary~\ref{cor:mu-vanish-together}.
\end{proof}

The following anticyclotomic analogue of \cite[Theorem 4.3.3]{epw} is an immediate corollary.

\begin{theorem}\label{thm:mu-in-families}
If $\cL$ is one of the conditions $\{\emptyset,0\}, \{\Gr,0\}, \{\Gr,\Gr\}, \{\Gr,\emptyset\}$, then the following are equivalent.
\begin{enumerate}
\item $\mu_\cL(f)=0$ for some $f \in \mathcal{\hat{H}}$.
\item $\mu_\cL(f)=0$ for every $f \in \mathcal{\hat{H}}$.
\end{enumerate}
\end{theorem}
\begin{proof}
This follows from Proposition~\ref{prop:mucong} and the first half of Lemma~\ref{lem:compareprimitive}.
\end{proof}

Suppose that $f_E \in \mathcal{\hat{H}}$ is a modular form of weight $k=2$ corresponding to an elliptic curve $E / \Q$. The following theorem is due to Matar \cite[Theorem 3.4]{matar}.

\begin{theorem}\label{thm:matar}
In addition to the hypotheses already present in this paper, assume that
\begin{enumerate}
\item $\mathrm{Gal}(\Q(E[p])/\Q) = \mathrm{GL}_2(\F_p)$
\item $p \nmid \# E(\F_p$)
\item $a_p \neq 2 \mod p$
\end{enumerate}
Then $\mu(f_E)=0.$
\end{theorem}

Thus, if $\mathcal{\hat{H}}$ contains a weight $2$ modular form $f_E$ which is associated to an elliptic curve $E / \Q$ satisfying the hypotheses of Theorem \ref{thm:matar}, then $\mu(f)=0$ for every $f \in \mathcal{\hat{H}}$ by Theorem \ref{thm:mu-in-families}. It would be interesting to try to adapt the arguments of \cite{matar} to more general modular forms, which would enlarge the number of immediate applications of the main results of the present paper. The authors hope to pursue this line of inquiry in future work.

\section{Comparing $\lambda$-invariants for congruent modular forms}\label{sec:main}

In this section, we compare the $\lambda$-invariants of several Selmer groups for  congruent modular forms under the assumption that the corresponding $\mu$-invariants vanish. As in \S\ref{sec:mu}, we do so on utilizing non-primitive Selmer groups.
It is clear from the definitions that we have exact sequences
\begin{equation}\label{eq:selmer-left-exact-Gr-empty}
0 \rightarrow \Sel_{\Gr,\emptyset}(K,\AA) \rightarrow H^1_{\Sigma}(K,\AA) \rightarrow \prod_{\substack{v \in \Sigma \setminus \Sigma_0 \\ v \nmid p}} H_v \times H_{\mathfrak{p}}
\end{equation}
and
\begin{equation}\label{eq:selmer-left-exact-0-empty}
0 \rightarrow \Sel_{{\emptyset,0}}(K,\AA) \rightarrow H^1_{\Sigma}(K,\AA) \rightarrow {\left(\prod_{\substack{v \in \Sigma \setminus \Sigma_0 \\ v \nmid p}} H_v\right) \times H^1(K_{\mathfrak{\bar p}},\AA)}.
\end{equation}
In fact, the sequences \eqref{eq:selmer-left-exact-Gr-empty} and  \eqref{eq:selmer-left-exact-0-empty} as well as \eqref{eq:nonprimitive-selmer-left-exact}  are all exact on the right for $\mathcal L=\{\Gr,\emptyset \}$ or $\{0,\emptyset\}$.

\begin{proposition}\label{prop:gr-zero-sel-exact}
The sequences
\[
0 \rightarrow \Sel_{\Gr,\emptyset}(K,\AA) \rightarrow H^1_{\Sigma}(K,\AA) \xrightarrow{\gamma_1} \left(\prod_{\substack{v \in \Sigma \setminus \Sigma_0 \\ v \nmid p}} H_v\right) \times H_{\mathfrak{p}} \rightarrow 0,
\]

\[
0 \rightarrow \Sel_{{\emptyset,0}}(K,\AA) \rightarrow H^1_{\Sigma}(K,\AA) \xrightarrow{\gamma_2}\left( \prod_{\substack{v \in \Sigma \setminus \Sigma_0 \\ v \nmid p}} H_v \right)\times H^1(K_{{\bar\p}},\AA)\rightarrow 0 ,
\]
and
\[
0 \rightarrow \Sel_{\mathcal L}(K,\AA) \rightarrow \Sel^{\Sigma_0}_{\mathcal L}(K,\AA) \rightarrow \prod_{v \in \Sigma_0} H_v \rightarrow 0,
\]
for $\mathcal L=\{\Gr,\emptyset \}$ or ${\{\emptyset,0\}}$, are all exact.
\end{proposition}
\begin{proof}
As in \cite[Proposition A.2]{pollackweston11}, this comes down to showing that the global-to-local maps
$\gamma_1$ and $\gamma_2$ are surjective. {But this follows from \cite[Proposition~3.2.1]{Greenberg11}, the hypotheses of which we verified in Lemmas \ref{lem:no-subquotient-mup} and \ref{lem:LEO-and-CRK}.}
\end{proof}

We will need the following two lemmas.

\begin{lemma}\label{lem:primitive-no-finite-index}
For  $\mathcal L=\{\Gr,\emptyset \}$ or ${\{\emptyset,0\}}$, the non-primitive modified Selmer group $\Sel^{\Sigma_0}_{\mathcal L}(K,\AA)$ has no proper finite-index $\Lambda$-submodules.
\end{lemma}
\begin{proof}
This follows from the proof of Proposition \ref{prop:modified-selmer-no-finite-index} and \cite[Proposition 4.2.1]{Greenberg16}.
\end{proof}

\begin{lemma}\label{lem:modified-same-lambda}
We have an equality of $\lambda$-invariants
\[
\lambda(\Sel_{\Gr,0}(K,\AA)) = \lambda(\Sel_{\Gr,\emptyset}(K,\AA)).
\]
\end{lemma}
\begin{proof}
This follows from Lemma \ref{lem:selmer-ranks}(iii).
\end{proof}

The following result has been proved in many other contexts under the assumption that $\Sel^{\Sigma_0}_{\mathcal L}(K,\AA)$ is $\Lambda$-cotorsion (cf \cite{greenbergvatsal,epw,kidwell16}). Using the results from Section \ref{sec:comparing}, we may dispense of that hypothesis, allowing us to include the case $\mathcal L= \{ \Gr,\emptyset \}$ in our analysis.

\begin{proposition}\label{prop:compare-non-primitive}
Let $f$ and $g$ be modular forms as given in Proposition~\ref{prop:mucong} and let $\mathcal L=\{\Gr,\emptyset \}$ or ${\{\emptyset,0\}}$. If $\mu ( \cS_{\mathcal L}(f) )=\mu ( \cS_{\mathcal L}(g))=0$, then $\lambda(\cS_{\mathcal L}(f))= \lambda(\cS_{\mathcal L}(g))$.
\end{proposition}
\begin{proof}
As in the proof of Proposition~\ref{prop:mucong}, we have the isomorphism
\[
\cS_{\mathcal L}(f)[\varpi] \simeq\cS_{\mathcal L}(g)[\varpi]
\]
and that $\cS_{\mathcal L}(f)$ and $\cS_{\mathcal L}(g)$ have the same $\Lambda$-corank.
By Lemma \ref{lem:primitive-no-finite-index}, neither $\cS_\cL(f)$ nor $\cS_\cL(g)$ has any finite-index submodules. Thus, the hypotheses of Proposition~\ref{prop:computelambda} are satisfied, hence the $\lambda$-invariant of $\cS_\cL(f)$ can be read from the $\Gamma_n$-coinvariants of Tate twists of $\cS_\cL(f)[\varpi]$, and similarly for $\cS_\cL(g)$. It follows that
\[
\lambda(\cS_{\mathcal L}(f))= \lambda(\cS_{\mathcal L}(g))
\]
as desired.
\end{proof}

With this lemma, we may now prove a comparison theorem for $\lambda$-invariants of the modified Selmer groups in the spirit of Greenberg-Vatsal. For a set of primes $\Sigma_0$, we write
\[
\delta(\Sigma_0,f) := \sum_{v \in \Sigma_0} \lambda(\mathcal H_v(\AA_f)).
\]
Note that this quantity is independent of any local conditions $\mathcal L$.

\begin{theorem}\label{thm:main-modified}
Let $f$ and $g$ be modular forms satisfying the hypotheses of Proposition \ref{prop:compare-non-primitive}. For $\mathcal L=\{\Gr,\emptyset \}$ or ${\{\emptyset,0\}}$, suppose that $\mu_\cL(f)=\mu_\cL(g)=0$. Then,
\begin{equation}\label{eq:lambda-diff}
\lambda_{\mathcal L}(f) - \lambda_{\mathcal L}(g) = \delta(\Sigma_0,f)-\delta(\Sigma_0,g).
\end{equation}
\end{theorem}
\begin{proof}
Let $\AA$ denote either $\AA_f$ or $\AA_g$. By Proposition \ref{prop:gr-zero-sel-exact}, we have an exact sequence
\[
0 \rightarrow \Sel_{\mathcal L}(K,\AA) \rightarrow \Sel^{\Sigma_0}_{\mathcal L}(K,\AA) \rightarrow \prod_{v \in \Sigma_0} H_v \rightarrow 0.
\]
Since $\prod_{v \in \Sigma_0} H_v$ is $\Lambda$-cotorsion by Lemma~\ref{lem:Hv-size}, we deduce from Proposition \ref{prop:additivity} that
\[
\lambda_{\mathcal L}(f) + \delta(\Sigma_0,f) = \lambda(\cS_{\mathcal L}(f))
\]
and
\[
\lambda_{\mathcal L}(g) + \delta(\Sigma_0,g) = \lambda(\cS_{\mathcal L}(g)).
\]
The result now follows from Proposition \ref{prop:compare-non-primitive}.
\end{proof}

\begin{remark}
Note that by Lemma \ref{lem:modified-same-lambda}, it is an additional corollary that \eqref{eq:lambda-diff} also holds for $\mathcal L=\{\Gr,0\}$, though the proof given above does not apply directly to these local conditions.
\end{remark}

\begin{remark} For any specific example, Lemma \ref{lem:Hv-size} can be used to compute $\delta(\Sigma_0,f)$ and $\delta(\Sigma_0,g)$.
\end{remark}

As a corollary to Theorem \ref{thm:main-modified}, we can obtain a theorem on the variation of $\lambda$-invariants for the Greenberg Selmer group. First, note that we have a natural restriction map
\[
\mathrm{loc}_{\mathfrak{p}} : \Sel(K,\TT) \rightarrow H^1_f(K_{\mathfrak{p}},\TT),
\]
and let us write $\mathrm{ck}(\TT)$ for the cokernel of this map; that is,
\[
\mathrm{ck}(\TT) := \frac{H^1_f(K_{\mathfrak{p}},\TT)}{\mathrm{loc}_{\mathfrak{p}} \Sel(K,\TT)}.
\]
As shown in the proof of \cite[Lemma A.3]{Castella}, $\mathrm{loc}_{\mathfrak{p}}$ is actually an injection, and $\mathrm{ck}(\TT_f)$ is $\Lambda$-cotorsion.

\begin{theorem}\label{thm:main-greenberg}
Retain the same hypotheses as in Theorem \ref{thm:main-modified}. Write $\mu(f)$ and $\lambda(f)$ for the $\mu$- and $\lambda$-invariants of the Greenberg Selmer group $\Sel(K,\AA_f)$, respectively, and similarly for $g$. Assume that $\mu_{\emptyset,\Gr}(f)=0$. Then $\mu(f)=\mu(g)=0$, and in this case we have
\begin{equation}\label{eq:lambda-diff-greenberg}
\lambda(f) - \lambda(g) = \left[ \delta(\Sigma_0,f) + \lambda\left( \mathrm{ck}(\TT_f)\right) \right] - \left[ \delta(\Sigma_0,g) + \lambda\left( \mathrm{ck}(\TT_g)\right) \right].
\end{equation}
\end{theorem}

\begin{proof}
By global duality, we have an exact sequence
\begin{equation}\label{eq:main-greenberg}
0 \rightarrow \mathrm{ck}(\TT_f) \rightarrow \mathcal X_{\emptyset,\Gr}(K,\AA_f) \rightarrow \mathcal X(K,\AA_f) \rightarrow 0.
\end{equation}
Since the first term is torsion, we have by Proposition \ref{prop:additivity}
\[
\mu_{\emptyset,\Gr}(f) = \mu(f) + \mu(\mathrm{ck}(\TT_f)),
\]
and since $\mu$-invariants are non-negative, we have $\mu(f) = \mu(\mathrm{ck}(\TT_f))=0$. By Theorem~\ref{thm:mu-in-families}, $\mu(g)=0$ as well. For the rest of the proof, assume $\mu(f)=\mu(g)=0$. Upon applying Proposition \ref{prop:additivity} again, we have the equalities
\[
\lambda_{\emptyset,\Gr}(f) = \lambda(f) + \lambda\left( \mathrm{ck}(\TT_f)\right)
\]
and
\[
\lambda_{\emptyset,\Gr}(g) = \lambda(g) + \lambda\left( \mathrm{ck}(\TT_g)\right).
\]
Since $\lambda_{\emptyset,\Gr}(f)=\lambda_{\emptyset,\Gr}(g)$ by Theorem \ref{thm:main-modified}, the result follows.
\end{proof}

\section{Relation to work of Kriz-Li}\label{sec:kriz-li}

Let $f_1$ and $f_2$ in $\mathcal{\hat{H}}$ be modular forms associated to elliptic curves $E_1$ and $E_2$ over $\Q$ with conductors $N_1$ and $N_2$, respectively. Thus, we are assuming that $E_1[p] \simeq E_2[p]$ as $G_\Q$-modules.  For $i=1,2$, fix modular paramaterizations $\pi_i:X_0(N_i) \rightarrow E_i$, and define $\omega_{E_i} \in H^0(E_i / \Q, \Omega^1)$ to be the invariant differential defined by
\[
\pi_i^\ast(\omega_{E_i})=f_{E_i} \frac{dq}{q}.
\]
Associated to each $\omega_{E_i}$ is a formal logarithm $\log_{\omega_{E_i}}$. Under they hypothese we have placed on our imaginary quadratic field $K / \Q$, there exist {\it Heegner points} $P_i \in E_i(K)$ for $i=1,2$. The following is a recent theorem of Kriz and Li \cite{KrizLi}.

\begin{theorem}\label{thm:KL}
If $p$ splits in $K$, then
\[
\left( \prod_{\ell \mid pN_1 N_2/M} \frac{|\tilde{E}^{ns}_1(\F_\ell)|}{\ell} \right) \cdot \log_{\omega_{E_1}}P_1 \equiv \pm \left( \prod_{\ell \mid pN_1 N_2/M} \frac{|\tilde{E}^{ns}_2(\F_\ell)|}{\ell} \right) \cdot \log_{\omega_{E_2}}P_2 \mod p \mathcal O,
\]
where
\[
M=\prod_{\substack{\ell \mid {(N_1,N_2)}  \\  a_\ell(E_1)\equiv a_\ell(E_2) \mod p }}   \ell^{\ord_\ell(N_1N_2)}
\]
Here $\tilde{E}^{ns}$ denotes the non-singular part of the mod $\ell$ reduction of $E$.
\end{theorem}

In fact, as mentioned in \cite[Remark 1.3]{KrizLi}, their methods show that a similar type of congruence is satisfied for more general modular forms of weights $k \geq 2$ upon replacing the $p$-adic logarithms of Heegner points with the $p$-adic Abel-Jacobi images of the generalized Heegner cycles of \cite{BDP}. In Section 3 of \cite{Castella}, Castella shows that, under some additional hypotheses (including $N^- > 1$, which removes us from the setting of the present paper), we have
\[
\Char_\Lambda(X_{\emptyset,0}(K,\AA_{f_E})) = (\mathcal L_p(f_E/K)),
\]
where $\mathcal L_p(f_E/K)$ is the $p$-adic $L$-function constructed in \cite{BDP} (c.f.  \cite[Theorem~3.4]{Castella}). Given that the $p$-adic logarithms of Heegner points  give special values of $L_p(f_E/K)$, our Theorem~\ref{thm:main-modified} for the choice of local conditions $\{\emptyset,0\}$ can therefore be viewed as an algebraic analogue of Theorem~\ref{thm:KL}.

\section*{Acknowledgments}
The authors would like to thank Ralph Greenberg, Matteo Longo, Bharathwaj Palvannan and Stefano Vigni for answering many of our questions during the preparation of this paper.  { We would also like to thank the referee for carefully reading an earlier version of the manuscript and for giving very  helpful comments and suggestions that have led to substantial improvements of the presentation of  the paper.}

\bibliographystyle{amsalpha}
\bibliography{references}
\end{document}